\documentclass[12pt,reqno]{amsart}
\usepackage[a4paper]{geometry}
\usepackage{graphicx}
\usepackage{float}
\usepackage{subcaption}

\theoremstyle{plain}
\newtheorem{theorem}{Theorem}[section]

\newtheorem{corollary}[theorem]{Corollary}

\theoremstyle{definition}
\newtheorem*{definition}{Definition}

\theoremstyle{remark}
\newtheorem*{remark}{Remark}

\numberwithin{theorem}{section}
\numberwithin{equation}{section}
\numberwithin{figure}{section}

\newcommand{\CC}{\mathbb C}
\newcommand{\RR}{\mathbb R}

\renewcommand{\Re}{\operatorname{Re}}
\renewcommand{\Im}{\operatorname{Im}}

\newcommand{\dd}{\mathrm{d}}
\newcommand{\ii}{\mathrm{i}}

\begin{document}
\title{The effect of repeated differentiation on $L$-functions}

\author{Jos Gunns}
\address{School of Mathematics, University of Bristol, University Walk, Bristol, BS8 1TW}
\email{jos.gunns@bristol.ac.uk}

\author{Christopher Hughes}
\address{Department of Mathematics, University of York, York, YO10 5DD, United Kingdom}
\email{christopher.hughes@york.ac.uk}

\date{1 May 2018}


\begin{abstract}
We show that under repeated differentiation, the zeros of the Selberg $\Xi$-function become more evenly spaced out, but with some scaling towards the origin. We do this by showing the high derivatives of the $\Xi$-function converge to the cosine function, and this is achieved by expressing a product of Gamma functions as a single Fourier transform.
\end{abstract}

\maketitle

\section{Introduction}

In 2006 Haseo Ki \cite{Ki05} proved a conjecture of Farmer and Rhoades \cite{FarRho}, that differentiating the Riemann $\Xi$-function evens out the zero spacing. Specifically Ki showed that there exists sequences $A_n$ and $C_n$ with  $C_n \to 0$ slowly such that
\begin{equation}\label{eq:Zeta_Cosine}
\lim_{n \to \infty}A_n\Xi^{(2n)}(C_nz)=\cos(z),
\end{equation}

In this paper we extend Ki's result to the entire Selberg Class of $L$-functions, showing that there exists sequences $A_n$ and $C_n$ (which depend on the properties of $L$-function under consideration) and constants $M'$ and $\theta$, such that
\begin{equation*}
\lim_{n \to \infty}A_n\Xi_F^{(2n)}\left(C_n z-\frac{M'}{\Lambda}\right)= \cos(z+\theta).
\end{equation*}
where $\Xi_F$ is the Xi-function for the $L$-function $F$, an element of the Selberg Class. This result is stated more precisely in Theorem~\ref{thm:diffL}.

In \cite{Selberg}, Selberg proposed an axiomatic definition of an $L$-function, now known as the Selberg Class.
\begin{definition}
A function $F(s)$ is an element of the Selberg Class if:
\begin{enumerate}
\item It has a Dirichlet series of the form
\[
F(s) = \sum_{n=1}^\infty \frac{a_n}{n^s}
\]
which is absolutely convergent for $\Re(s)>1$.

\item It is a meromorphic function such that $(s-1)^m F(s)$ is an entire function of order 1, for some integer $m\geq 0$.

\item It has a functional equation of the form $\Phi(s) = \overline{\Phi(1-\overline{s})}$, where
\[
\Phi(s) =  \epsilon Q^s F(s) \prod_{j=1}^k \Gamma(\lambda_j s +\mu_j)
\]
with $\epsilon, Q, \lambda_j$ and $\mu_j$ all constants, and subject to $|\epsilon|=1$, $Q>0$, $\lambda_j>0$ and $\Re(\mu_j)\geq 0$.

\item The coefficients in the Dirichlet series satisfy $a_1=1$ and $a_n = O(n^\delta)$ for some fixed positive $\delta$.

\item It has an Euler product in the sense that
\[
\log F(s) = \sum_{n} \frac{b_n}{n^s}
\]
with $b_n=0$ unless when $n=p^r$ for some prime $p$ and $r$ a positive integer, and $b_n = O(n^\theta)$ for some $\theta<1/2$.
\end{enumerate}
\end{definition}

Kaczorowski and Perelli \cite{KP11} define an Extended Selberg Class, essentially by dropping the requirement for the function to satisfy an Euler product. Our results apply equally to elements of this extended class of $L$-functions.
\begin{definition}
A function $F(s)$ is an element of the Extended Selberg Class if it satisfies axioms (1)--(3) above.
\end{definition}

\begin{remark}
The degree of an $L$-function is $2\Lambda$, where
\[
\Lambda = \sum_{j=1}^k \lambda_j .
\]
It is conjectured that the degree is always an integer. However, this is only known for $L$-functions of degree 2 or less \cite{KP11}. More specifically, it is believed that, using the duplication formula, the gamma functions can be transformed so that $\lambda_j=1/2$ for all $j$ (and in such a case, the $L$-function has degree $k$).
\end{remark}

\begin{definition}
Let $F$ be an element of the Selberg Class, and set
\[
\xi_F(s) = s^m (1-s)^m \epsilon Q^s \prod_{j=1}^k \Gamma(\lambda_j s +\mu_j) F(s).
\]
\end{definition}

Note that by assumption of $F$ being in the Selberg Class, $\xi_F(s)$ is an entire function of order 1, with the functional equation $\xi_F(s) = \overline{\xi_F(1-\overline{s})}$.

\begin{definition}
Set $\Xi_F(z) = \xi_F(\tfrac12+\ii z)$.
\end{definition}

\begin{remark}
From the functional equation $\Xi_F(z)$ is a real function for $z\in\RR$. If the Dirichlet coefficients of $F$ are real, then $\Xi(z)$ is an even function.
\end{remark}

Ki proved his result for the Riemann $\Xi$-function by starting with the integral representation of the Gamma function to show that
\begin{equation*}
\Xi_\zeta(z)=\int_{-\infty}^{\infty}\varphi(x)e^{\ii x z}\dd x,
\end{equation*}
where
\begin{equation*}
\varphi(x)=2 \sum_{n=1}^\infty \left(2n^4 \pi^2 e^{9x/2} - 3n^2\pi e^{5x/2} \right) e^{-n^2\pi e^{2 x}}.
\end{equation*}
Note that the functional equation yields the fact that $\varphi(x) = \varphi(-x)$.

After a suitable change of variables, this yields
\begin{equation*}
\Xi_\zeta(z)=2\pi^2 \int_{0}^{\infty} e^{-a e^x} e^{b x} \left(1+O(e^{-x})\right) \left(e^{\ii x z/2} + e^{-\ii x z/2}\right)\dd x,
\end{equation*}
with $a=\pi$ and $b=9/4$. By differentiating such integrals, Ki was able to explicitly show the existence of sequences $A_n$ and $C_n$ such that \eqref{eq:Zeta_Cosine} held. His method also holds for Hecke $L$-functions, since the functional equation, analogously to the Riemann Xi-function, can be written with a single Gamma function. However, the Selberg Class of $L$-functions generally includes a product of disparate Gamma functions, which cannot be simplified down to a single one by the multiplication formula of the Gamma function.

In sections \ref{sect:FourierTrans} and \ref{sect:RepeatedDiff}, we find the Fourier transform for the analogous $\Xi$-function for an element of the (extended) Selberg Class of $L$-functions, showing it can be written as
\begin{equation*}
\Xi_F(z)=B \int_{-\infty}^{\infty} \varphi(x)e^{\ii \Lambda z x}\dd x,
\end{equation*}
where $\varphi(x) = e^{-a e^{x}} e^{b x} \left(1+O(e^{-x})\right)$ as $x\to\infty$, and where $\Lambda = \sum\lambda_j$.

In section \ref{sect:RepeatedDiff}, we start from that result to demonstrate the existence of sequences $A_n$ and $C_n$ such that
\begin{equation*}
\lim_{n \to \infty} A_n \Xi_F^{(2n)}\left(C_n z-\frac{M'}{\Lambda}\right)=\cos(z+\theta)
\end{equation*}
where $\theta=\arg(B)$ and  $M'=\sum_{j=1}^k \Im \mu_j$. We utilize a similar argument to that used by Ki.

The rates of convergence are considered in section \ref{sect:graphs}, demonstrated by numerical examples.

\subsection*{Acknowledgements}

This research forms part of the first author's PhD thesis from the University of York \cite{Gunns17}.


\section{Expressing the $\Xi$-function as a Fourier transform}\label{sect:FourierTrans}

\begin{theorem}\label{thm:FourierTransformXi}
Let $F$ be an element of the Selberg Class, with data $m$, $k$, $\varepsilon, Q, \lambda_j$, and $\mu_j$. The Fourier transform of the $Xi$-function related to $F$ is
\begin{align*}
\widehat\Xi_F(x) &= \int_{-\infty}^\infty \Xi_F(z) e^{-\ii x z} \dd z\\
&= \hat B  \exp\left(- \hat a e^{x/\Lambda}+ \hat b x \right) \left(1+O\left(e^{-x/\Lambda}\right)\right)
\end{align*}
where
\[
\hat a =  \Lambda Q^{-1/\Lambda}  \prod_{j=1}^k \lambda_j^{-\lambda_j / \Lambda}
\]
and
\[
\hat b= \frac{2m+M+\tfrac12\Lambda}{\Lambda}
\]
and
\[
\hat B =  (-1)^m \varepsilon Q^{-(M+2m)/\Lambda}  (2\pi)^{(k+1)/2} \Lambda^{2m-1/2} \prod_{j=1}^k \lambda_j^{-\frac12 + \mu_j + \lambda_j \left(  - M -2m\right)/\Lambda}
\]
where
\[
\Lambda= \sum_{j=1}^k \lambda_j
\]
and
\[
M = \sum_{j=1}^k \mu_j - \tfrac12(k-1).
\]
\end{theorem}

\begin{remark}
Note that $\Lambda$ and $M$ are invariant under the Gamma multiplication formulae.
\end{remark}

Recall that
\begin{align*}
\Xi_F(z)&=\xi_F(\frac{1}{2}+iz) \\
&=\varepsilon Q^{1/2+iz} \left(\tfrac14+z^2\right)^m F(\tfrac{1}{2}+\ii z) \prod_{j=1}^k\Gamma(\ii \lambda_jz+\mu_j+\tfrac{1}{2}\lambda_j)
\end{align*}
is an entire function.
We wish to find its Fourier transform
\[
\widehat\Xi_F(x) = \int_{-\infty}^\infty \Xi_F(z) e^{-\ii x z} \dd z.
\]
Shifting the contour so that $F(s)$ can be represented by its Dirichlet series, swapping the order of summation and integration and shifting the contour back, we find that
\begin{equation}\label{eq:XiHat}
\widehat\Xi_F(x) = \varepsilon Q^{1/2} \sum_{n=1}^\infty \frac{a_n}{n^{1/2}} \int_{-\infty}^\infty \left(\tfrac14+z^2\right)^m \prod_{j=1}^k\Gamma(\ii \lambda_jz+\mu_j+\tfrac{1}{2}\lambda_j) \left(\frac{n e^x}{Q}\right)^{-\ii z} \dd z.
\end{equation}
Thus the Fourier transform can be found by convolutions and differentiations of the Fourier transform of the Gamma function.

\begin{theorem}[Fourier transform of multiple gamma functions]\label{thm:FT gamma}
Let $\lambda_1,\dots,\lambda_k>0$ and let $\alpha_1,\dots,\alpha_k$ be such that their real parts are all positive. Then for large $T$,
\begin{multline*}
\int_{-\infty}^\infty \left(\prod_{j=1}^k \Gamma(\alpha_j + \ii\lambda_j z) \right) e^{-\ii T z} \dd z \\
=
C_k \exp\left(-\Lambda e^{T/\Lambda} \prod_{j=1}^k \lambda_j^{-\lambda_j / \Lambda} + \frac{T(A-(k-1)/2)}{\Lambda}\right) \left(1+O\left(e^{-T/\Lambda}\right)\right)
\end{multline*}
where $\Lambda= \sum_{j=1}^k \lambda_j$ and $A = \sum_{j=1}^k \alpha_j$ and
\begin{equation}\label{eq:Ck}
C_k = \frac{(2\pi)^{(k+1)/2}}{\sqrt{\Lambda}} \prod_{j=1}^k \lambda_j^{-\frac12 + \alpha_j + \lambda_j \left( \frac12 (k-1) - A \right)/\Lambda}.
\end{equation}
\end{theorem}

\begin{remark}
Booker stated a similar result in the case when $\lambda_j=1/2$ for all $j$, in section 5.2 of \cite{Boo05}.
\end{remark}

\begin{proof}
We prove this theorem by induction. The base case, when $k=1$ says that for $\lambda>0$ and $\Re(\alpha)>0$,
\begin{equation}\label{eq:gamma_FT}
\int_{-\infty}^\infty \Gamma(\ii \lambda z+\alpha) e^{-\ii T z} \dd z =  \frac{2\pi}{\lambda} \exp\left(-e^{T/\lambda}+ T\alpha/\lambda \right).
\end{equation}
This is simply the Fourier transform of one gamma function, a classical result.

With our choice of Fourier constants the convolution theorem is
\[
\int_{-\infty}^\infty f(z) g(z) e^{-\ii T z} \dd z = \frac{1}{2\pi} \int_{-\infty}^\infty \widehat f(x) \widehat g(T-x) \dd x
\]
where $\widehat f$ and $\widehat g$ are the Fourier transforms of $f$ and $g$ respectively. The Fourier transform of $k+1$ gamma functions will be the convolution of the Fourier transform of $k$ gamma functions with the Fourier transform of one gamma function, both of which are known by the inductive hypothesis. That is,
\begin{multline}\label{eq:conv_many_gammas}
\int_{-\infty}^\infty \left(\prod_{j=1}^{k+1} \Gamma(\alpha_j + \ii\lambda_j z) \right) e^{-\ii T z} \dd z \\
 =
\frac{C_k}{\lambda_{k+1}} \int_{-\infty}^\infty  \exp\left(-\Lambda e^{x/\Lambda} \prod_{j=1}^k \lambda_j^{-\lambda_j / \Lambda} + \frac{x(A-(k-1)/2)}{\Lambda}\right) \left(1+O\left(e^{-x/\Lambda}\right) \right) \\
\times \exp\left(-e^{(T-x)/\lambda_{k+1}}+ \frac{(T-x)\alpha_{k+1}}{\lambda_{k+1}} \right) \dd x
\end{multline}
where we have set $\Lambda= \sum_{j=1}^k \lambda_j$ and $A = \sum_{j=1}^k \alpha_j$. Later in the proof, we will also set $\Lambda'= \sum_{j=1}^{k+1} \lambda_j$ and $A' = \sum_{j=1}^{k+1} \alpha_j$.

We will asymptotically evaluate this integral. Note that the exponential in the integrand is dominated by
\[
 -\Lambda e^{x/\Lambda} \prod_{j=1}^k \lambda_j^{-\lambda_j / \Lambda} -e^{(T-x)/\lambda_{k+1}}
\]
and this has a maximum at $x=x_0$ where $x_0$ is such that
\[
 - e^{x_0/\Lambda} \prod_{j=1}^k \lambda_j^{-\lambda_j / \Lambda} + \frac{1}{\lambda_{k+1}} e^{(T-x_0)/\lambda_{k+1}} = 0
\]
that is
\[
x_0 =  \frac{T \Lambda}{\Lambda'} + \frac{\lambda_{k+1} \Lambda}{\Lambda'} \ln\left(\frac{1}{\lambda_{k+1}} \prod_{j=1}^k \lambda_j^{\lambda_j / \Lambda} \right)
\]
where $\Lambda' = \Lambda+\lambda_{k+1} = \sum_{j=1}^{k+1} \lambda_j$.

Thus, expanding around $x=x_0+\epsilon$ for small $\epsilon$, we have (after a fair amount of straightforward algebraic simplification, and using the identity $\Lambda' = \Lambda + \lambda_{k+1}$)
\begin{multline*}
- \Lambda e^{x/\Lambda} \prod_{j=1}^k \lambda_j^{-\lambda_j / \Lambda} -e^{(T-x)/\lambda_{k+1}}  = -e^{T/\Lambda'} \prod_{j=1}^{k+1} \lambda_j^{-\lambda_j / \Lambda'} \left( \Lambda e^{\epsilon / \Lambda} + \lambda_{k+1} e^{-\epsilon/\lambda_{k+1}} \right) \\
= - \Lambda' e^{T/\Lambda'} \prod_{j=1}^{k+1} \lambda_j^{-\lambda_j / \Lambda'} \left(1 + \frac{1}{2\lambda_{k+1} \Lambda} \epsilon^2  + B_1 \epsilon^3 + O( \epsilon^4) \right)
\end{multline*}
where $B_1$ is an inconsequential constant that depends upon $\Lambda$ and $\lambda_{k+1}$. (We remark that it is no surprise the coefficient of the $\epsilon$ term is zero, as this is the expansion around the maximum of the LHS).

Substituting $x=x_0+\epsilon$ in the two other terms in the exponent of the integrand in \eqref{eq:conv_many_gammas} and letting $A' = A + \alpha_{k+1} = \sum_{j=1}^{k+1} \alpha_j$ we have
\begin{multline*}
\frac{x(A-\tfrac12(k-1))}{\Lambda} + \frac{(T-x)\alpha_{k+1}}{\lambda_{k+1}} = \frac{T (A'-\tfrac12(k-1))}{\Lambda'} \\
+ \frac{\lambda_{k+1} (A-\tfrac12(k-1)) - \alpha_{k+1}\Lambda}{\Lambda'} \ln\left(\frac{1}{\lambda_{k+1}} \prod_{j=1}^k \lambda_j^{\lambda_j / \Lambda} \right)
+ B_2 \epsilon
\end{multline*}
where $B_2 = \frac{A-\tfrac12(k-1)}{\Lambda} - \frac{\alpha_{k+1} }{\lambda_{k+1}} $ is another inconsequential constant.

Substituting both these expansions back into \eqref{eq:conv_many_gammas} we see that the Fourier transform of the $k+1$ Gamma functions is asymptotically
\begin{multline*}
C \exp\left(- \Lambda' e^{T/\Lambda'} \prod_{j=1}^{k+1} \lambda_j^{-\lambda_j / \Lambda'} + \frac{T (A'-\tfrac12(k-1))}{\Lambda'} \right) \\
\times  \int
\exp\left( - \epsilon^2\frac{\Lambda'}{2\lambda_{k+1} \Lambda} e^{T/\Lambda'} \prod_{j=1}^{k+1} \lambda_j^{-\lambda_j / \Lambda'} \left(1 + B_1 \epsilon + O( \epsilon^2)\right) + B_2 \epsilon \right) \dd \epsilon
\end{multline*}
where
\begin{equation}\label{eq:C}
C = \frac{ C_k }{\lambda_{k+1}} \left(\frac{1}{\lambda_{k+1}} \prod_{j=1}^k \lambda_j^{\lambda_j / \Lambda} \right)^{ \frac{\lambda_{k+1} (A-\frac12(k-1)) - \alpha_{k+1}\Lambda}{\Lambda'} }.
\end{equation}

We utilise here the normal methods of asymptotic analysis, where the range of the $\epsilon$ integral is thought of as being small (so $O(\epsilon)$ terms are small), but $\epsilon^2 e^{T/\Lambda'}$ is large, so the Gaussian integral can be extended to the whole real line with trivially small error. To be concrete, truncate the $\epsilon$ integral to be over $\left[-e^{-T/3\Lambda'} , e^{-T/3\Lambda'} \right]$ and let $Q=\frac{\Lambda'}{2\lambda_{k+1} \Lambda} \prod_{j=1}^{k+1} \lambda_j^{-\lambda_j / \Lambda'}$, so we have
\begin{multline*}
\int_{-e^{-T/3\Lambda'}}^{e^{-T/3\Lambda'}} e^{- \epsilon^2 Q e^{T/\Lambda'}  \left(1 + B_1 \epsilon + O( \epsilon^2)\right) + B_2 \epsilon } \dd \epsilon\\
= \int_{-e^{-T/3\Lambda'}}^{e^{-T/3\Lambda'}} e^{ - \epsilon^2 Q e^{T/\Lambda'}}
\left(1 - B_1 Q e^{T/\Lambda'} \epsilon^3  + B_2 \epsilon + O\left(e^{2T/\Lambda'} \epsilon^6\right) \right) \dd \epsilon.
\end{multline*}
We can extend the integral to be over all $\RR$ with a tiny error, of size $O\left( e^{-Q e^{T/3\Lambda'}}\right)$. Note that due to the symmetry of the integral, the odd terms in $\epsilon$ vanish, and note that the big-O term in the integrand contributes $O\left(e^{-3T/2\Lambda'}\right)$ to the integral. Therefore, the above integral equals
\begin{multline*}
 \int_{-\infty}^\infty e^{ - \epsilon^2 Q e^{T/\Lambda'}}  \left(1 + O\left(e^{2T/\Lambda'} \epsilon^6\right) \right) \dd \epsilon + O\left( e^{-Q e^{T/3\Lambda'}}\right)\\
=\sqrt{\frac{\pi}{Q}} e^{-T/2\Lambda'} \left(1 + O(e^{-T/\Lambda'})\right).
\end{multline*}

It is easy to see the contribution to \eqref{eq:conv_many_gammas} from outside the range
\begin{equation*}
\left[x_0-e^{-T/3\Lambda'} , x_0+e^{-T/3\Lambda'} \right]
\end{equation*}
contributes a tiny amount, dominated by the error term above, and so
\begin{multline*}
\int_{-\infty}^\infty \left(\prod_{j=1}^{k+1} \Gamma(\alpha_j + \ii\lambda_j z) \right) e^{-\ii T z} \dd z  = \sqrt{\frac{2\pi \lambda_{k+1} \Lambda}{\Lambda'}} \prod_{j=1}^{k+1} \lambda_j^{\lambda_j / (2\Lambda')} C \times\\
\times \exp\left(- \Lambda' e^{T/\Lambda'} \prod_{j=1}^{k+1} \lambda_j^{-\lambda_j / \Lambda'} + \frac{T (A'-\tfrac12 k)}{\Lambda'} \right) \left(1+O\left(e^{-T/\Lambda'}\right) \right).
\end{multline*}
In order to simplify the constant, recall the definitions of $C$ given in $\eqref{eq:C}$ and $C_k$ given in \eqref{eq:Ck}. After some rearranging, we see that
\begin{align*}
\sqrt{\frac{2\pi \lambda_{k+1} \Lambda}{\Lambda'}} \prod_{j=1}^{k+1} \lambda_j^{\lambda_j / (2\Lambda')} C
&= \frac{(2\pi)^{(k+2)/2}}{\sqrt{\Lambda'}} \prod_{j=1}^{k+1} \lambda_j^{-1/2 + \alpha_k + \lambda_j \left(k/2 - A' \right)/\Lambda'} \\
&=C_{k+1}
\end{align*}
which is the  required form for $k+1$ Gamma functions, thus completing the proof.
\end{proof}

\begin{corollary}
Let $\lambda_1,\dots,\lambda_k>0$ and let $\alpha_1,\dots,\alpha_k$ be such that their real parts are all positive. Then for large $T$,
\begin{multline*}
\int_{-\infty}^\infty \left(\tfrac14 + z^2\right)^m \left(\prod_{j=1}^k \Gamma(\alpha_j + \ii\lambda_j z) \right) e^{-\ii T z} \dd z \\
=
C_{k,m} \exp\left(-\Lambda e^{T/\Lambda} \prod_{j=1}^k \lambda_j^{-\lambda_j / \Lambda} + \frac{T(2m+A-(k-1)/2)}{\Lambda}\right) \left(1+O\left(e^{-T/\Lambda}\right)\right)
\end{multline*}
where $\Lambda= \sum_{j=1}^k \lambda_j$ and $A = \sum_{j=1}^k \alpha_j$ and
\begin{equation*}
C_{k,m} = (-1)^m (2\pi)^{(k+1)/2} \Lambda^{2m-1/2} \prod_{j=1}^k \lambda_j^{-\frac12 + \alpha_j + \lambda_j \left( \frac12 (k-1) - A -2m\right)/\Lambda}.
\end{equation*}
\end{corollary}

\begin{proof}
The new term $\left(\tfrac14 + z^2\right)^m$ requires the first $2m$ derivatives of the RHS to be calculated. The big-O term is differentiable, and note that it dominates all the derivatives other than the $2m$\textsuperscript{th} derivative. The result then follows immediately.
\end{proof}

\begin{proof}[Proof of Theorem~\ref{thm:FourierTransformXi}]
First note that from the above Corollary, the contribution to \eqref{eq:XiHat} for the terms with $n>1$ are exponentially smaller than the error term in $n=1$ term, for large $x$. Since $a_1=1$ for an element of the Selberg Class, we have that for large $x$,
\begin{multline*}
\widehat \Xi_F(x) = (-1)^m \varepsilon Q^{-(M+2m)/\Lambda}  (2\pi)^{(k+1)/2} \Lambda^{2m-1/2} \prod_{j=1}^k \lambda_j^{-\frac12 + \mu_j - \lambda_j \left( M +2m\right)/\Lambda} \\
\times \exp\left(-\Lambda Q^{-1/\Lambda} e^{x/\Lambda} \prod_{j=1}^k \lambda_j^{-\lambda_j / \Lambda} + (2m+M+\tfrac12\Lambda)\frac{x}{\Lambda}\right) \left(1+O\left(e^{-x/\Lambda}\right)\right),
\end{multline*}
where we have used the Corollary above, with $\alpha_j = \mu_j+\tfrac12\lambda_j$, $T=x-\log Q$ and we set $M = \sum_{j=1}^k \mu_j - \tfrac12(k-1)$. This is the theorem, with the constants $\hat B$, $\hat a$ and $\hat b$ given explicitly.
\end{proof}

\begin{remark}
The proof made essential use of only the first three assumptions arising from $F(s)$ being an element of the Selberg class. Therefore this result holds for $F$ an element of the Extended Selberg Class (with $\hat B$ being trivially changed if $a_1\neq1$).
\end{remark}


\section{The $\Xi$-function under repeated differentiation}\label{sect:RepeatedDiff}

Note that with our choice of constants, the inverse Fourier transform is
\begin{equation*}
\Xi_F(z) = \frac{1}{2\pi} \int_{-\infty}^\infty \widehat\Xi_F(x) e^{\ii x z} \dd x.
\end{equation*}

Note that the $\mu_j$, part of the data of the $L$-function $F$, could be complex. If we define
\begin{equation*}
    M'=\sum_{j=1}^k \Im \mu_j,
\end{equation*}
and rescale $z$ we have
\begin{align*}
\Xi_F\left(\frac{z-M'}{\Lambda}\right) &=\frac{\Lambda}{2\pi} \int_{-\infty}^\infty \widehat\Xi_F(x \Lambda) e^{-\ii x M'} e^{\ii x z } \dd x \\
&= B \int_{-\infty}^{\infty} \varphi(x)e^{ixz} \dd x
\end{align*}
where by Theorem~\ref{thm:FourierTransformXi}
\begin{equation}\label{eq:phi}
\varphi(x) = e^{-a e^x} e^{b x} \left(1+O(e^{-x})\right),
\end{equation}
with
\begin{equation}\label{eq:a}
a =  \Lambda Q^{-1/\Lambda}  \prod_{j=1}^k \lambda_j^{-\lambda_j / \Lambda},
\end{equation}
\begin{equation}\label{eq:b}
b= 2m+\tfrac12\Lambda-\tfrac12(k-1) + \sum_{j=1}^k \Re \mu_j
\end{equation}
and $B = \hat B \Lambda / 2\pi$. (Note that $a,b \in\RR$ and, in the notation of Theorem~\ref{thm:FourierTransformXi},   $a=\hat a$ and $b = \Lambda \hat b - \ii M'$).

\begin{theorem}\label{thm:diffL}
Let $\Xi_F(z)$ be the Xi-function for the $L$-function $F$, an element of the Selberg Class. Let $w_n$ be defined as the solution to
\begin{equation*}
    a w_n e^{w_n} = b w_n +2n
\end{equation*}
where $a$ and $b$ are given by \eqref{eq:a} and \eqref{eq:b} respectively. Then uniformly on compact subsets of $\CC$,
\begin{equation*}
\lim_{n\to\infty} A_n \Xi_F^{(2n)} \left(C_n z-\frac{M'}{\Lambda}\right) = \cos(z+\arg(B)),
\end{equation*}
where  $\Lambda$,  $M'$, and $B$ are given in Theorem~\ref{thm:FourierTransformXi}, and the sequences $A_n$ and $C_n$ are given by
\begin{equation*}
    A_n = (-1)^n \exp\left(a e^{w_n} - b w_n\right)  \frac{\sqrt{n}}{2|B|\Lambda^{2n} w_n^{2n+1/2} \sqrt{\pi}}
\end{equation*}
and
\begin{equation*}
    C_n = \frac{1}{\Lambda w_n}.
\end{equation*}
\end{theorem}

\begin{remark}
One can see that, for large $n$, the $w_n$ defined in the theorem satisfies
\begin{equation*}
    w_n\sim \log\left(\frac{2n}{a}\right)-\log\log\left(\frac{2n}{a}\right) .
\end{equation*}
\end{remark}

\begin{proof}
From the functional equation for the $L$-function we have that
\begin{equation*}
\Xi_F\left(\frac{z-M'}{\Lambda}\right)=\overline{\Xi_F\left(\frac{\overline{z}-M'}{\Lambda}\right)}
\end{equation*}
so
\begin{align*}
B \int_{-\infty}^{\infty}\varphi(x)e^{ixz}\dd x&=\overline{B}\int_{-\infty}^{\infty}\varphi(x)e^{-ixz}\dd x\\
&=\overline{B}\int_{-\infty}^{\infty}\varphi(-x)e^{ixz}\dd x,
\end{align*}
and since this holds for any $z \in \mathbb{C}$ we have
\begin{equation*}
    B\varphi(x)=\overline{B}\varphi(-x).
\end{equation*}
Therefore
\begin{equation}\label{eq:Xi_as_f_integral}
    \Xi_F\left(\frac{z-M'}{\Lambda}\right)=\int_0^{\infty}\varphi(x)\left(B e^{ixz}+\overline{B}e^{-ixz}\right) \dd x.
\end{equation}
We can now consider just the integral
\begin{equation*}
    f(z)=\int_0^{\infty}\varphi(x)e^{ixz}\dd x
\end{equation*}
as the second integral will behave in much the same way. Differentiating this, we have that
\begin{equation*}
    f^{(2n)}(z)=(-1)^n \int_0^{\infty}\varphi(x)x^{2n}e^{ixz}\dd x.
\end{equation*}

Haseo Ki \cite{Ki05} proved that uniformly on compact subsets of $\CC$,
\begin{equation*}
    \lim_{n \to \infty}\int_0^{\infty}v_n\varphi(w_nx)x^{2n}e^{ixz}\dd x =e^{iz},
\end{equation*}
where $\varphi(x)$ is of the form given in \eqref{eq:phi}, and $w_n$ is defined such that
\begin{equation*}
    a w_n e^{w_n} = b w_n +2n
\end{equation*}
and
\begin{equation*}
    v_n = \sqrt{\frac{n w_n}{\pi}}e^{ ae^{w_n}} e^{-b w_n} .
\end{equation*}

Therefore, we have that
\begin{align*}
    f^{(2n)}(z / w_n)&=(-1)^n\int_0^{\infty}\varphi(x)x^{2n}e^{\ii x z / w_n}\dd x\\
    &=(-1)^n w_n^{2n+1}\int_0^{\infty}\varphi(w_nx)x^{2n}e^{ixz}\dd x
\end{align*}
and using Ki's work (and including the error term) we have
\begin{equation*}
    f^{(2n)}(z/w_n)=\sqrt{\frac{\pi}{nw_n}}(-1)^ne^{-ae^{w_n}}e^{bw_n}w_n^{2n+1}e^{iz}\left(1+\mathcal{O}(w_n^{-2})\right).
\end{equation*}
From \eqref{eq:Xi_as_f_integral} we see that
\begin{equation*}
\frac{1}{\Lambda^{2n}} \Xi_F^{(2n)}\left(\frac{z-M'}{\Lambda}\right) = B  f^{(2n)}(z) + \overline{B}f^{(2n)}(-z)
\end{equation*}
so setting $C_n = \frac{1}{\Lambda w_n}$,
\begin{align*}
(-1)^n e^{a e^{w_n} - b w_n} w_n^{-2n-1} \sqrt{\frac{n w_n}{\pi}} \frac{1}{|B| \Lambda^{2n}}  &  \Xi_F^{(2n)}\left(C_n z-\frac{M'}{\Lambda}\right) \\
    &=\left(\frac{B}{|B|} e^{iz}+\frac{\overline{B}}{|B|}e^{-iz}\right)(1+\mathcal{O}(w_n^{-2})) \\
    &=2\cos(z+\arg(B))(1+\mathcal{O}(w_n^{-2}))
\end{align*}
and after taking the limit, the proof Theorem~\ref{thm:diffL} is complete.
\end{proof}


\section{Numerical Demonstrations}\label{sect:graphs}

In this section we briefly discuss how the $L$-function's data affects the  convergence to the cosine function. Recall that the error term is $O(w_n^{-2})$
where
\begin{equation*}
w_n \sim \log\left(\frac{2n}{a}\right),
\end{equation*}
with
\begin{equation*}
a = \Lambda Q^{-1/\Lambda}  \prod_{j=1}^k \lambda_j^{-\lambda_j / \Lambda}.
\end{equation*}

Therefore $L$-functions with larger conductor converge slightly more quickly, and high degree $L$-functions converge more slowly. This fact is clearer if one assumes that one can transform the $L$-function so its data has $\lambda_j=1/2$ for all $j$, since then $a = k Q^{-2/k}$.

The sequence $C_n$ effectively scales the density of the zeros of the $(2n)$\textsuperscript{th} derivative. We have that
\begin{equation*}
C_n=\frac{1}{\Lambda w_n} \to 0.
\end{equation*}
which means that the zeros of the unscaled $(2n)$\textsuperscript{th} derivative have moved towards the origin. Compare, for example, the Riemann Xi-function before any derivatives have been taken and  after 100 derivatives have been taken.
\begin{figure}[H]
\begin{subfigure}{.49\textwidth}
\includegraphics[width=.8\linewidth]{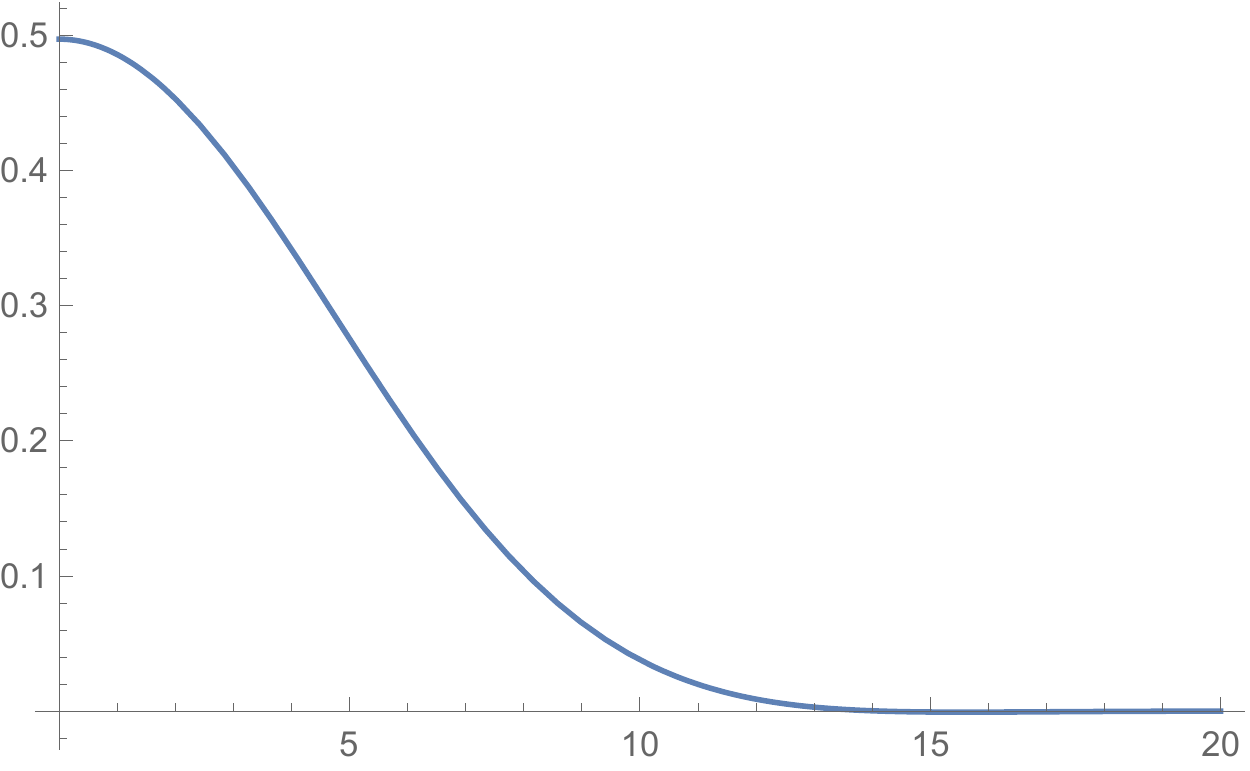}
\caption{$n=0$}
\end{subfigure}
\begin{subfigure}{.49\textwidth}
\includegraphics[width=.8\linewidth]{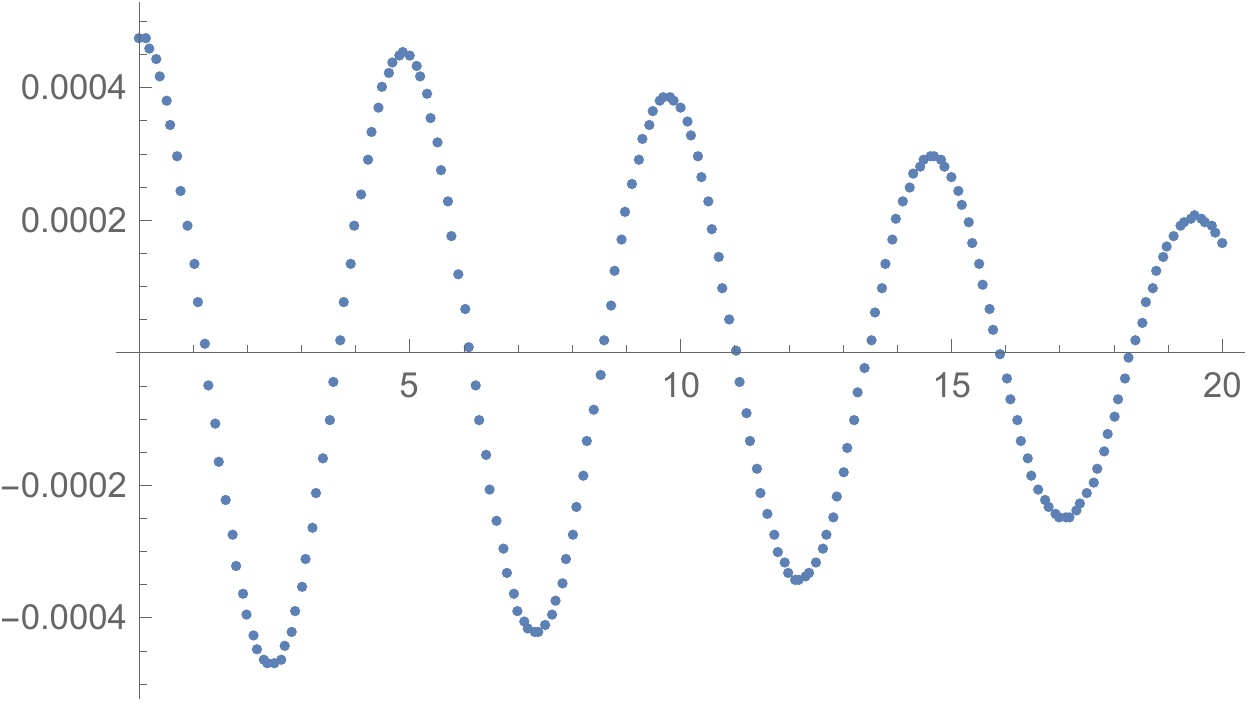}
\caption{$n=50$}
\end{subfigure}
\caption{Plots of the Riemann Xi-function after $2n$ derivatives}
\end{figure}
These figures also demonstrate the convergence to the cosine function.

Finally, the $A_n$ term dictates how large the derivatives of the $L$-functions get. From
\begin{equation*}
A_n=\frac{\sqrt{n}(-1)^ne^{ae^{w_n}}e^{-bw_n}}{2w_n^{2n+1/2}\sqrt{\pi}|B| \Lambda^{2n}}
\end{equation*}
and using the defining equation for $w_n$, $a w_n e^{w_n} = b w_n + 2n$, we have that
\begin{equation*}
\log \left| A_n \right| = 2n(1-\log\Lambda - \log w_n) - a e^{w_n} ( w_n-1 ) + \tfrac12\log n - \tfrac12\log w_n + O(1)
\end{equation*}
and so since $w_n \sim \log(2n/a)$, as $n \to \infty$ we have that $A_n \to 0$, which means that the size of the $(2n)$\textsuperscript{th} derivative gets large as $n$ increases, although for $L$-functions of small degree where $\log \Lambda < 1$ the size of the derivatives can initially decrease, before eventually increasing.

\end{document}